\documentclass[11pt]{article}%
\usepackage{amsfonts}
\usepackage{amssymb}
\usepackage{amsmath}
\usepackage{graphicx}
\usepackage{theorem}
\usepackage[a4paper]{geometry}
\usepackage[toc]{appendix}
\usepackage{hyperref}
\usepackage{makeidx}%
\theoremstyle{break}
\newtheorem{theorem}{Theorem}[section]

\newtheorem{lemma}[theorem]{Lemma}

\newtheorem{proposition}[theorem]{Proposition}
\theorembodyfont{\upshape}
\newtheorem{remark}[theorem]{Remark}

\newenvironment{proof}[1][Proof]{\textbf{#1.} }{\ \rule{0.5em}{0.5em}}
\numberwithin{equation}{section}
\makeindex
\begin{document}

\title{Universal structures in some mean field spin glasses, and an application}
\author{Erwin Bolthausen\thanks{Universit\"{a}t Z\"{u}rich, eb@math.uzh.ch}%
\thanks{Supported in part by the Swiss National Foundation under contract no
200020-116348} \quad Nicola Kistler\thanks{ENS Lyon,
nkistler@umpa.ens-lyon.fr}}
\date{}
\maketitle

\begin{abstract}
We discuss a spin glass reminiscent of the Random Energy Model, which allows
in particular to recast the Parisi minimization into a more classical Gibbs
variational principle, thereby shedding some light on the physical meaning of
the order parameter of the Parisi theory. As an application, we study the
impact of an extensive cavity field on Derrida's REM: Despite its simplicity,
this model displays some interesting features such as ultrametricity and chaos
in temperature.

\end{abstract}

\section{Introduction}

After years of intensive research and important advances (\cite{ass},
\cite{guerra}, \cite{guerratoninelli}, \cite{talagrand}), the Parisi theory
\cite{parisi}, originally developed in the study of the
Sherrington-Kirkpatrick model of spin glasses, still remains mathematically
quite elusive.

Despite the spectacular proof by Guerra and Talagrand that Parisi's replica
symmetry breaking scheme provides the correct free energy for the SK-model,
many aspects of the Parisi ansatz continue to present major challenges. In
fact, the appearance of seemingly universal features, such as the
Derrida-Ruelle hierarchical structures, the (related) ultrametricity, the law
of the pure states, are still far from being understood.

We hope to gain some modest insights into these issues by considering
generalizations of the Random Energy Model (REM for short), that is, models
with Hamiltonians given by independent random variables. Our generalization is
different from the \textquotedblleft generalized random energy
model\textquotedblright\ invented by Derrida. It can be analyzed by large
deviation techniques. Despite its simplicity, it exhibits a number of
interesting properties, like asymptotic ultrametricity, Poisson-Dirichlet
description of the pure states, chaos in temperature, and a non-trivial
dependence of the overlap structure on the temperature. The free energy is
given by a Parisi-type formula which naturally can be linked to a Gibbs
variational formula via a kind of duality relation which makes apparent why an
infimum appears in the Parisi formulation.

The second part of this work presents a particular mean field spin glass which
we call the "REM+Cavity". It is related to the random overlap structures of
Aizenman, Sims and Starr \cite{ass}; but, instead of taking the
thermodynamical limit in the REM \textit{first} and \textit{then} perform a
one spin perturbation of the Derrida-Ruelle structures, we perform a cavity
field perturbation on the finite systems first, and only subsequently do we
take the thermodynamical limit. As a first step, we stick here to the simplest
finite size counterpart of the Derrida-Ruelle structures, the Random Energy
Model \cite{derridaREM}. Our model shows a delicate phase transition where
Replica Symmetry is broken and ultrametricity sets in. In the low temperature
region massive pure states emerge, with law being given by the
Poisson-Dirichlet distribution. The model also displays chaotic behavior in
temperature. The natural extensions of our approach to models with more
intricate dependencies than those of REM-type (such as for instance the
Generalized Random Energy Models) turns out to be quite a subtle. We will
address this issue in a forthcoming paper.

\section{Mean field models of REM-type\label{Sect_REMType}}

Consider a double sequence $X_{\alpha,i},\alpha,i\geq1,$ of i.i.d. random
variables with a distribution $\mu$, taking values in a Polish space
$(S,{\mathcal{S}})$, and which are defined on a probability space
$(\Omega,{\mathcal{F}},{\mathbb{P}})$. For $N\in{\mathbb{N}}$, $\alpha>1$, the
empirical distributions is defined by
\[
L_{N,\alpha}\overset{\mathrm{def}}{=}{\frac{1}{N}}\sum_{i=1}^{N}%
\delta_{X_{\alpha,i}},
\]
which takes values in ${{\mathcal{M}}}_{1}^{+}(S)$, the set of probability
measures on $(S,{\mathcal{S}})$, which itself is a Polish space when equipped
with the weak topology. Let $\Phi:{{\mathcal{M}}}_{1}^{+}\rightarrow
{\mathbb{R}}$ be a continuous function. We write
\[
Z_{N}\overset{\mathrm{def}}{=}2^{-N}\sum_{\alpha=1}^{2^{N}}\operatorname{exp}%
\left[  N\Phi(L_{N,\alpha})\right]  ,\;f_{N}(\Phi,\mu)\overset{\mathrm{def}%
}{=}{\frac{1}{N}}\log Z_{N}.
\]

\begin{theorem}
\label{main_theorem_abstract}The limit $f(\Phi,\mu)=\lim_{N\rightarrow\infty
}f_{N}(\Phi,\mu)$ exists ${\mathbb{P}}-a.s.,$ is non-random, and is given as%
\[
f(\Phi,\mu)=\sup\left\{  \Phi(\nu)-H(\nu\mid\mu):\;H(\nu\mid\mu)\leq
\log2\right\}  ,
\]
where $H$ is the usual relative entropy $H(\nu\mid\mu)\overset{{\text{def}}%
}{=}\int\log\left(  {\frac{d\nu}{d\mu}}\right)  d\nu$ if $\nu\ll\mu$ and
$\log(d\nu/d\mu)\in L_{1}(\mu)$, and $=\infty$ otherwise.
\end{theorem}

We specialize to linear functionals $\Phi\left(  \nu\right)  =\int\phi d\nu$,
$\phi:S\rightarrow{\mathbb{R}},$ i.e.%
\begin{equation}
Z_{N}=2^{-N}\sum_{\alpha}\operatorname{exp}\left[  \sum_{i=1}^{N}%
\phi(X_{\alpha,i})\right] \label{PartFu_linear}%
\end{equation}
In order that $\Phi$ is continuous, we have to assume that $\phi$ is bounded
and continuous, a condition we want to relax somewhat. By a slight abuse of
notation, we write $f_{N}(\phi,\mu)$ for the free energy of the finite-size
system, and $f(\phi,\mu)$ for its limit, which by Theorem
\ref{main_theorem_abstract} is given through
\begin{equation}
f(\phi,\mu)=\sup\left\{  \int\phi(x)\nu(dx)-H(\nu\mid\mu):\;H(\nu\mid\mu
)\leq\log2\right\}  ,\label{gvp}%
\end{equation}
at least if $\phi$ is bounded and continuous. We shall refer to expression
\eqref {gvp} as the Gibbs variational principle (GVP). Let us write for a
distribution $\nu\in{{\mathcal{M}}}_{1}^{+}(S)$, and $h:S\rightarrow
{\mathbb{R}}$, $E_{\nu}[h]\overset{{\text{def}}}{=}\int h(x)\nu(dx)$, and for
$m\in{\mathbb{R}},\,\Gamma_{\phi}(m)\overset{\mathrm{def}}{=}\log E_{\mu
}\left[  \mathrm{e}^{m\phi}\right]  $, which we always assume to exist. We
also define the probability measure $G_{m}$ on $S$ by%
\begin{equation}
\frac{dG_{m}}{d\mu}\overset{\mathrm{def}}{=}\frac{\mathrm{e}^{m\phi}}{Z\left(
m\right)  },\label{Def_Gm}%
\end{equation}
$Z\left(  m\right)  $ is the appropriate norming constant.

\begin{theorem}
\label{solving_gvp} Assume $\phi:S\rightarrow\mathbb{R}$ is continuous, and
satisfies%
\begin{equation}
\int\mathrm{e}^{\lambda\phi}d\mu<\infty\label{Cond_ExponentialMoment}%
\end{equation}
for all real $\lambda.$ Then%
\begin{equation}
\lim_{N\rightarrow\infty}f_{N}\left(  \phi,\mu\right)  =f\left(  \phi
,\mu\right)  ,\label{VarFormula_LinearUnbounded}%
\end{equation}
$f$ given by (\ref{gvp}). Furthermore, there exists a unique maximizer of the
right hand side of (\ref{gvp}) in the form $G_{m_{\ast}}$ where $m_{\ast}%
\in(0,1]$ is characterized as follows: If
\begin{equation}
\Gamma_{\phi}^{\prime}(1)-\Gamma_{\phi}(1)\leq\log2,\label{high_temperature}%
\end{equation}
then $m_{\star}=1$. Otherwise $m_{\star}\in(0,1)$ is the unique solution to
the following equation:
\begin{equation}
m\Gamma_{\phi}^{\prime}(m)-\Gamma_{\phi}(m)=\log2.\label{entropy_condition}%
\end{equation}

\end{theorem}

If $m_{\ast}=1,$ i.e. (\ref{high_temperature}) holds, we say the model is in
\textit{high temperature}, and otherwise in \textit{low temperature}.

For the Sherrington-Kirpatrick model the free energy was originally obtained
by Parisi using the replica method, and a special ansatz for the so-called
\textquotedblleft replica symmetry breaking\textquotedblright. The physical
content of Parisi's functional is still somewhat mysterious despite of
considerable progress made later. In our setting, the nonrigorous
RSB-mechanism would yield the following free energy for a spin glass of the
form (\ref{PartFu_linear}) as
\begin{equation}
\operatorname{Parisi}\left(  \phi,\mu\right)  \overset{\mathrm{def}}%
{=}\operatorname{inf}_{m\in\lbrack0,1]}\left\{  {\frac{\log2}{m}}+{\frac{1}%
{m}}\log E_{\mu}\mathrm{e}^{m\phi}-\log2\right\}  ,\label{FreeEnergy_Parisi}%
\end{equation}
The fact that one takes the infimum instead of the usual supremum in the Gibbs
formalism is at first sight rather puzzling. However, in our setup, the
identification of (\ref{gvp}) with the right-hand side of
(\ref{FreeEnergy_Parisi}) will be rather straightforward, and we have

\begin{theorem}
\label{equiv_PVP_GVP}
\[
f(\phi,\mu)=\operatorname{Parisi}\left(  \phi,\mu\right)
\]

\end{theorem}

We learned from Guerra \cite{guerrayep} a simple argument how to prove that
$f\left(  \phi,\mu\right)  $ is bounded by (\ref{FreeEnergy_Parisi}): For
$m\in\lbrack0,1]$,
\begin{align*}
f_{N}(\phi,\mu)  & ={\frac{1}{N}}\log\left(  2^{-N}\sum_{\alpha}\exp\left[
\sum\nolimits_{i}\phi(X_{\alpha,i})\right]  \right) \\
& ={\frac{1}{mN}}\log\left(  2^{-N}\sum_{\alpha}\exp\left[  \sum
\nolimits_{i}\phi(X_{\alpha,i})\right]  \right)  ^{m}\\
& \leq{\frac{1}{mN}}\log\left(  2^{-mN}\sum_{\alpha}\exp\left[  m\sum
\nolimits_{i}\phi(X_{\alpha,i})\right]  \right)  ,
\end{align*}
where the last bound follows by straightforward convexity/concavity arguments.
Taking expectation w.r.t. the randomness, exploiting concavity of the
logarithm, the independence of the random variables appearing in the sum
$\sum_{i}\phi(X_{\alpha,i})$, and optimizing over $m\in\lbrack0,1]$ one easily
gets
\begin{equation}
{\mathbb{E}}f_{N}(\phi,\mu)\leq\operatorname{Parisi}\left(  \phi,\mu\right)
,\label{upper_parisi}%
\end{equation}
uniformly in $N$. We will not use that, and we will give another proof of Theorem
\ref{equiv_PVP_GVP} in Section \ref{SubSect_FreeEnergy}.

As usual, the Gibbs measure is defined by%
\[
{{\mathcal{G}}}_{\Phi,N}(\alpha)\overset{\mathrm{def}}{=}\frac{2^{-N}%
\operatorname{exp}\left[  N\Phi(L_{N,\alpha})\right]  }{Z_{N}},\ 1\leq
\alpha\leq2^{N}.
\]
We analyze this only in the linear case $\Phi\left(  \nu\right)  =\int\phi
d\nu.$ By an abuse of notation, we write it simply as ${{\mathcal{G}}}%
_{\phi,N}(\alpha).$

We recall the definition of the Poisson-Dirichlet point process with parameter
$m\in\left(  0,1\right)  .$ We first consider a Poisson point process on
$\mathbb{R}^{+}$ with intensity measure $t^{-m-1}dt.$ We call such a point
process a $\operatorname{PP}\left(  m\right)  .$ This point process has
countably many single points with a maximal element. If we order the points
downwards, we obtain a sequence of random variables $\xi_{1}>\xi_{2}>\cdots$ .
If $m<1,$ then $\zeta\overset{\mathrm{def}}{=}\sum_{i}\xi_{i}<\infty,$ almost
surely, and we can define $\overline{\xi}_{i}\overset{\mathrm{def}}{=}\xi
_{i}/\zeta.$ Then $\sum_{i}\delta_{\overline{\xi}_{i}} $ is a
Poisson-Dirichlet point process with parameter $m.$ We write
$\operatorname{PD}\left(  m\right)  $ for such a point process.

\begin{theorem}
\label{limiting_gibbs_measure} Suppose that $\phi,\mu$ are such that the
system is in low temperature, i.e. $m_{\ast}<1.$ ($m_{\ast}$ the unique
solution to the entropy condition (\ref{entropy_condition})). Assume
furthermore that the distribution of $\phi$ under $\mu$ is non-lattice. Then
the point process $\sum_{\alpha}\delta_{{{\mathcal{G}}}_{\phi,N}(\alpha)}$
converges weakly as $N\rightarrow\infty$ to a $\operatorname{PD}\left(
m_{\star}\right)  $.
\end{theorem}

Remark that Theorem \ref{limiting_gibbs_measure} accounts for some
universality of the Derrida-Ruelle structures and the so-called
Poisson-Dirichlet distribution, which naturally arise in the weak limits of
the Gibbs measure associated to a REM-system in low temperature.

\section{The REM+Cavity model\label{REM+Cavity}}

We give an application of the previous results. Let again $N\in{\mathbb{N}}$.
We set $\Sigma_{N}\overset{{\text{def}}}{=}\{1,\dots,2^{N}\}$ and consider on
some probability space $(\Omega,{\mathcal{F}},{\mathbb{P}})$ a sequence
$(X_{\alpha},\alpha\in\Sigma_{N})$ of independent, centered Gaussians with
variance $N$, as well as another independent sequence $(g_{\alpha,i},\alpha
\in\Sigma_{N},i=1,\dots,N)$ of independent standard Gaussians. For $\alpha
\in\Sigma_{N},\sigma=(\sigma_{1},\dots,\sigma_{N})\in\{\pm1\}^{N}$, we define
the Hamiltonian of the REM+Cavity-model.%
\begin{equation}
H(\alpha,\sigma)\overset{{\text{def}}}{=}X_{\alpha}+\sum_{i=1}^{N}g_{\alpha
,i}\sigma_{i}.\label{hamiltonian_remcavity}%
\end{equation}
$H(\cdot,\cdot)$ is thus a Gaussian field on $\Sigma_{N}\times\{\pm1\}^{N}$
with covariance given by
\[
{\mathbb{E}}[H(\alpha,\sigma)H(\alpha,\sigma^{\prime})]=N\delta_{\alpha
=\alpha^{\prime}}+N\delta_{\alpha=\alpha^{\prime}}q(\sigma,\sigma^{\prime}),
\]
where $q(\sigma,\sigma^{\prime})\overset{{\text{def}}}{=}{\frac{1}{N}}%
\sum_{i=1}^{N}\sigma_{i}\sigma_{i}^{\prime}$ is the usual overlap of the
configurations $\sigma,\sigma^{\prime}$. For $\beta\in{\mathbb{R}}$, the
inverse of the temperature, we define the free energy%
\begin{equation}
f_{N}(\beta)\overset{{\text{def}}}{=}{\frac{1}{N}}\log\left[  2^{-2N}%
\sum_{\alpha,\sigma}\operatorname{exp}\left(  \beta H(\alpha,\sigma)\right)
\right]  .\label{free_energy_remcavity}%
\end{equation}

\begin{proposition}
\label{propvariational} The limit $f(\beta)=\lim_{N\rightarrow\infty}%
f_{N}(\beta)$ exists ${\mathbb{P}}$-a.s. and is given by%
\begin{equation}
f(\beta)=\left\{
\begin{array}
[c]{cc}%
{\beta^{2}} & \mathrm{if\ }\beta\leq\beta_{\mathrm{cr}}\\
{\frac{\beta^{2}}{2}}m_{\star}+\frac{E[\cosh(\beta g)^{m_{\star}}\log
\cosh(\beta g)]}{E[\cosh(\beta g)^{m_{\star}}]}-{\log2} & \mathrm{if\ }%
\beta>\beta_{\mathrm{cr}}%
\end{array}
\right. \label{FreeEnergy_REMC}%
\end{equation}
with $\beta_{\mathrm{cr}}>0$ being the unique positive solutions of the
equation
\begin{equation}
E[\cosh(\beta g)\log\cosh(\beta g)]=\mathrm{e}^{\beta^{2}/2}\log
2,\label{critical_temperature}%
\end{equation}
and for $\beta>\beta_{\mathrm{cr}},$ $m_{\star}=m_{\star}(\beta)\in(0,1)$ is
the unique solution of
\begin{equation}
{\frac{\beta^{2}}{2}}m^{2}-\log E[\cosh(\beta g)^{m}]+m\frac{E[\cosh(\beta
g)^{m}\log\cosh(\beta g)]}{E[\cosh(\beta g)^{m}]}=\log
2.\label{entropy_condition_remcavity}%
\end{equation}

\end{proposition}

The mechanism behind this formula is easy to understand. Remark first that we
can write the partition function as%
\begin{align*}
2^{-2N}\sum_{\alpha,\sigma}\operatorname{exp}\left(  \beta H(\alpha
,\sigma)\right)   & =2^{-N}\sum_{\alpha}\mathrm{e}^{\beta X_{\alpha}}%
2^{-N}\sum_{\sigma}\exp\left[  \beta\sum\nolimits_{i=1}^{N}g_{\alpha,i}%
\sigma_{i}\right] \\
& =2^{-N}\sum_{\alpha}\mathrm{e}^{\beta X_{\alpha}}\prod\limits_{i=1}^{N}%
\cosh\left(  \beta g_{\alpha,i}\right) \\
& =2^{-N}\sum_{\alpha}\exp\left[  \beta X_{\alpha}+\sum\nolimits_{i=1}^{N}%
\log\cosh\left(  \beta g_{\alpha,i}\right)  \right] \\
& =2^{-N}\sum_{\alpha}\exp\left[  \beta X_{\alpha}+N\int\log\cosh\left(
\beta\right)  L_{N,\alpha}\left(  dx\right)  \right]  ,
\end{align*}
where%
\[
L_{N,\alpha}\overset{\mathrm{def}}{=}\frac{1}{N}\sum_{i=1}^{N}\delta
_{g_{\alpha,i}}.
\]
The probability that for a fixed $\alpha,$ we have $X_{\alpha}\approx yN,$ and
$L_{N,\alpha}\approx\nu$%
\[
\approx\exp\left[  -Ny^{2}/2-NH\left(  \nu\mid\mu\right)  \right], 
\]
$\mu$ being here the standard normal distribution. Arguing roughly in the
same way as before, we conclude that%
\begin{align}
f(\beta)  & =\sup_{y,\nu}\Big\{\beta y+\int\log\cosh\left(  \beta y\right)
\nu\left(  dy\right)  -y^{2}/2-H\left(  \nu\mid\mu\right)
\label{Varproblem_REM&Cav}\\
& \hspace{7cm} :y^{2}/2+H\left(  \nu\mid\mu\right)  \leq\log2\Big\},\nonumber
\end{align}
which leads to the expression in Proposition \ref{propvariational}.
$\beta_{\mathrm{cr}}$ is the value for which the restriction $y^{2}/2+H\left(
\nu\mid\mu\right)  \leq\log2$ becomes relevant in the supremum. An interesting
feature is that for any $\beta,$ the $\alpha$'s which are giving the main
contribution to the partition function are those for which%
\[
L_{N,\alpha}\approx\nu,
\]
$\nu$ being the maximizer in the variational problem. We will give a precise
derivation in Section \ref{SubSect_GibbsREMC}.

According to the convention following Theorem \ref{solving_gvp}, we call the
region $\beta\leq\beta_{\mathrm{cr}}$ the high-temperature, and $\beta
>\beta_{\mathrm{cr}}$ the low temperature regime. The associated Gibbs measure
is:%
\[
{\mathcal{G}}_{\beta,N}(\alpha,\sigma)={\frac{\operatorname{exp}\beta
H(\alpha,\sigma)}{\sum_{\alpha^{\prime},\sigma^{\prime}}\operatorname{exp}%
\beta H(\alpha^{\prime},\sigma^{\prime})}},\;{\text{for}}(\alpha,\sigma
)\in\Sigma_{N}\times\{\pm1\}^{N}.
\]
It is not difficult to realize that even in low temperature, the Gibbs weights
of individual configurations are exponentially small in $N.$ To get a
macroscopic weight we must lump together exponentially many configurations. In
the present situation, we have to take the marginal measure on the first
component:%
\[
{\mathcal{G}}_{\beta,N}^{\left(  1\right)  }(\alpha)\overset{{\text{def}}}%
{=}\sum_{\sigma\in\{\pm1\}^{N}}{\mathcal{G}}_{\beta,N}(\alpha,\sigma).
\]

\begin{proposition}
\label{point_process} If $\beta>\beta_{\mathrm{cr}}$, then the point process
$\sum_{\alpha}\delta_{{\mathcal{G}}_{\beta,N}^{\left(  1\right)  }(\alpha)}$
converges weakly to $\operatorname{PD}\left(  m_{\ast}\right)  $.
\end{proposition}

We thus witness in the low-temperature regime of the REM+Cavity the emergence
of massive pure states, with law being given by the Poisson-Dirichlet distribution.

We can also derive the limiting behavior of the overlaps under the replicated
Gibbs measure ${{\mathcal{G}}}_{\beta,N}^{\otimes2}$. Following the physicists
convention, we write $\left\langle \cdot\right\rangle _{\beta,N}^{\otimes2}$
for the expectation with respect to ${{\mathcal{G}}}_{\beta,N}^{\otimes2}.$
From the above proposition, it is clear that in the $N\rightarrow\infty$
limit, ${{\mathcal{G}}}_{\beta,N}^{\otimes2}\left(  \alpha=\alpha^{\prime
}\right)  $ has the same distribution as $\sum\eta_{i}^{2},$ where the
$\eta_{i}$ are the points of a $\operatorname{PD}\left(  m_{\ast}\right)  .$
Here $\alpha,\alpha^{\prime}$ are the first components of the two replicas.
The expectation of $\sum\eta_{i}^{2}$ is well known to be $1-m_{\ast}.$
Therefore, we get%
\begin{equation}
\lim_{N\rightarrow\infty}\mathbb{E}{{\mathcal{G}}}_{\beta,N}^{\otimes2}\left(
\alpha=\alpha^{\prime}\right)  =1-m_{\ast}.\label{Overlap_REMC}%
\end{equation}
Conditioned on $\alpha\neq\alpha^{\prime},$ the overlap of $\sigma
,\sigma^{\prime}$ is $0,$ in the $N\rightarrow\infty$ limit, whereas
conditioned on $\alpha=\alpha^{\prime},$ it is given by%
\begin{equation}
q_{\star}\overset{{\text{def}}}{=}{\frac{E\left[  \tanh^{2}(\beta
g)\operatorname{exp}\left(  m_{\star}\log\cosh(\beta g)\right)  \right]
}{E\left[  \operatorname{exp}(m_{\star}\log\cosh(\beta g))\right]  }%
},\label{qstar}%
\end{equation}
for $g$ a standard Gaussian and $E$ denoting expectation with respect to it.
To phrase it precisely

\begin{proposition}
[Ultrametricity for the REM+Cavity]\label{ultrametricity} For $\beta
>\beta_{\mathrm{cr}}$,%
\begin{equation}
\lim_{N\rightarrow\infty}{\mathbb{E}}\left[  \left\langle \delta
_{\alpha=\alpha^{\prime}}\left(  q(\sigma,\sigma^{\prime})-q_{\star}\right)
^{2}\right\rangle _{\beta,N}^{\otimes2}\right]  =0,\label{overlap1}%
\end{equation}
and%
\begin{equation}
\lim_{N\rightarrow\infty}{\mathbb{E}}\left[  \left\langle \delta_{\alpha
\neq\alpha^{\prime}}q\left(  \sigma,\sigma^{\prime}\right)  ^{2}\right\rangle
_{\beta,N}^{\otimes2}\right]  =0.\label{overlap2}%
\end{equation}

\end{proposition}

It should be remarked that our REM+Cavity model is \textit{not }%
ultrametrically structured for finite $N,$ in contrast to the pure REM or the
GREM. This means that the natural $L_{2}$-metric on the Gaussian Hamiltonians
is \textit{not }an ultrametric on $\Sigma_{N}\times\left\{  -1,1\right\}
^{N}.$

An interesting feature of our model is that it exhibits the so-called
\textquotedblleft chaos in temperature\textquotedblright, in sharp contrast
with the pure REM which does not have this property. The effect is easy to
understand. For a temperature parameter $\beta>\beta_{\mathrm{cr}},$
${{\mathcal{G}}}_{\beta,N}^{\left(  1\right)  }$ picks from the $\alpha$ for
which $X_{\alpha}\approx y_{\beta}^{\ast},\ L_{N,\alpha}\approx\nu_{\beta
}^{\ast},$ $\left(  y_{\beta}^{\ast},\nu_{\beta}^{\ast}\right)  $ being the
maximizer of (\ref{Varproblem_REM&Cav}). $y_{\beta}^{\ast},\nu_{\beta}^{\ast}$
depend in a non-trivial way on $\beta.$ In particular, they change when
$\beta$ is changed, regardless how large $\beta$ is. Therefore the
contribution to the partition function is coming from a completely different
set of $\alpha$'s if one changes the temperature parameter. This is in
contrast to the REM where for $\beta$ above the critical parameter, the
$\alpha$'s which contribute are always those for which the $X_{\alpha}$ are
close to the maximal possible value.

To phrase the property precisely, we have the following result:

\begin{proposition}
[Chaos in temperature for the REM+Cavity]\label{chaos_temperature} Assume
$\beta,\beta^{\prime}>\beta_{\star}$ and $\beta\neq\beta^{\prime}$. There
exists $\delta>0$ such that
\begin{equation}
{\mathbb{P}}\left[  {\mathcal{G}}_{\beta,N}\otimes{\mathcal{G}}_{\beta
^{\prime},N}(\delta_{\alpha=\alpha^{\prime}})\geq\mathrm{e}^{-\delta
N}\right]  \leq\mathrm{e}^{-\delta N},\label{chaos1}%
\end{equation}
and%
\begin{equation}
\lim_{N\rightarrow\infty}{\mathbb{E}}\left\langle \delta_{\alpha\neq
\alpha^{\prime}}q\left(  \sigma,\sigma^{\prime}\right)  ^{2}\right\rangle
_{\beta,\beta^{\prime},N}^{\otimes2}=0.\label{chaos2}%
\end{equation}

\end{proposition}

Summarizing, we have the following situation for the $N\rightarrow\infty$ the
Gibbs measure at $\beta>\beta_{\mathrm{cr}}$: It gives macroscopic weights to
$\alpha$'s for which $L_{N,\alpha}$ is approximately $\nu_{\beta}^{\ast}.$ The
random weights are given by a Poisson-Dirichlet point process with parameter
$m_{\ast}.$ If in the replicated system, $\alpha\neq\alpha^{\prime},$ then the
corresponding $\sigma,\sigma^{\prime}$ have zero overlap with probability
$\approx1.$ On the other hand, if $\alpha=\alpha^{\prime},$ then also the
$\sigma,\sigma^{\prime}$ have a non-zero overlap, given by $q_{\ast}.$ If
$\beta$ changes, then the choice is made from a completely different group of
$\alpha$'s.

\section{Proofs of the main results\label{proofs}}

\subsection{The free energy for spin glasses of
REM-type\label{SubSect_FreeEnergy}}

For the proof of Theorem \ref{main_theorem_abstract}, a technical result is
needed. Given $A\subset{{\mathcal{M}}}_{1}^{+}(S)$, we set
\[
M_{N}(A)\overset{\mathrm{def}}{=}\#\left\{  \alpha\leq2^{N}:L_{N,\alpha}\in
A\right\}  .
\]
We also write $H(A)$ for $\operatorname{inf}_{\nu\in A}H(\nu\mid\mu)$.

\begin{lemma}
\label{abstract_ldp} Let $\nu\in{\mathcal{M}}_{1}^{+}(S)$, and $V$ be an open
neighborhood of $\nu$. If $H(\nu\mid\mu)\leq\log2$, and $\varepsilon>0$, then
there exists an open neighborhood $U$ of $\nu$, $U\subset V$, and $\delta>0$
such that for large enough $N$
\begin{equation}
{\mathbb{P}}\Big [M_{N}(U)\leq\operatorname{exp}\left[  N(\log2-H(\nu\mid
\mu)-\varepsilon)\right]  \Big ]\leq\mathrm{e}^{-N\delta},\label{ldp_upper}%
\end{equation}%
\begin{equation}
{\mathbb{P}}\Big [M_{N}(U)\geq\operatorname{exp}\left[  N(\log2-H(\nu\mid
\mu)+\varepsilon)\right]  \Big ]\leq\mathrm{e}^{-N\delta}.\label{ldp_lower}%
\end{equation}
If $H(\nu)>\log2$, then there exist $U$ and $\delta$ as above, with
\[
{\mathbb{P}}\Big [M_{N}(U)\neq0\Big ]\leq\mathrm{e}^{-N\delta}.
\]

\end{lemma}

\begin{proof}
Let first $\nu\in{\mathcal{M}}_{1}^{+}(S)$ satisfy $H(\nu\mid\mu)\leq\log2$.
The statement of the Lemma is trivial if $\nu=\mu$, so we assume $\nu\neq\mu$.
Let $B_{r}(\nu)\subset{\mathcal{M}}_{1}^{+}(S)$ be the open ball of radius $r$
and center $\nu$, where we have equipped ${\mathcal{M}}_{1}^{+}(S)$ with one
of the standard metrics, e.g. Prohorov's metric. Then $H(B_{r}(\nu
))=H(\operatorname*{cl}(B_{r}(\nu)))$, except for countably many $r$.
Therefore we can find arbitrary small $\varepsilon_{1}>0$, and $U\overset
{\mathrm{def}}{=}B_{r}(\nu)\subset V$, such that $H(U)=H(\operatorname*{cl}%
(U))=H(\nu\mid\mu)-\varepsilon_{1}$. From Sanov's Theorem, we have
\[
{\mathbb{P}}\left(  L_{N,\alpha}\in U\right)  \geq\operatorname{exp}\left[
-N\left(  H(\nu\mid\mu)-{\frac{5}{6}}\varepsilon_{1}\right)  \right]  ,
\]%
\[
{\mathbb{P}}\left(  L_{N,\alpha}\in\operatorname*{cl}(U)\right)
\leq\operatorname{exp}\left[  -N\left(  H(\nu\mid\mu)-{\frac{7}{6}}%
\varepsilon_{1}\right)  \right]
\]
for large enough $N$. Therefore%
\[
{\mathbb{E}}M_{N}(U)\geq\operatorname{exp}\left[  N\left(  \log2-H(\nu\mid
\mu)+{\frac{5}{6}}\varepsilon_{1}\right)  \right]  ,
\]%
\[
{\mathbb{E}}M_{N}(U)\leq{\mathbb{E}}M_{N}(\operatorname*{cl}U)\leq
\operatorname{exp}\left[  N\left(  \log2-H(\nu\mid\mu)+{\frac{7}{6}%
}\varepsilon_{1}\right)  \right]  ,
\]
and using the independence of the $L_{N,\alpha}$%
\[
{\mathbb{E}}M_{N}(U)^{2}\leq\left(  {\mathbb{E}}M_{N}(U)\right)
^{2}+\operatorname{exp}\left[  N\left(  \log2-H(\nu\mid\mu)+{\frac{7}{6}%
}\varepsilon_{1}\right)  \right]  ,
\]%
\begin{equation} \label{variance_estimate}
\operatorname*{var}\nolimits_{{\mathbb{P}}}M_{N}(U)\leq\mathrm{e}%
^{-N\varepsilon_{1}/2}\left(  {\mathbb{E}}M_{N}(U)\right)  ^{2}%
\end{equation}
Hence,%
\[
{\mathbb{P}}\Big [M_{N}(U)\leq\left(  1-\mathrm{e}^{-N\varepsilon_{1}%
/8}\right)  {\mathbb{E}}M_{N}(U)\Big ]\leq\operatorname{exp}\left[
-N\varepsilon_{1}/4\right]  ,
\]%
\[
{\mathbb{P}}\Big [M_{N}(U)\geq\left(  1+\mathrm{e}^{-N\varepsilon_{1}%
/8}\right)  {\mathbb{E}}M_{N}(U)\Big ]\leq\operatorname{exp}\left[
-N\varepsilon_{1}/4\right]  .
\]
Choosing $\varepsilon_{1}$ smaller than $\varepsilon/2$ and $\delta
=\varepsilon_{1}/4$ proves the Lemma in this case. The case $H(\nu\mid
\mu)>\log2$ needs only a slight modification. In that case, there exists an
open neighborhood $U\subset V$ of $\nu$ such that ${\mathbb{P}}\left(
L_{N,\alpha}\in U\right)  $ is exponentially small in $N$, with a decay rate
which is bigger than $\log2$. This proves the claim by the Markov inequality.
\end{proof}

\begin{proof}
[Proof of Theorem \ref{main_theorem_abstract}]We first prove the lower bound.
Let $\nu$ be any element in ${\mathcal{M}}_{1}^{+}(S)$ satisfying $H(\nu
\mid\mu)\leq\log2$. Let $\varepsilon>0$. As $\Phi$ is continuous, we can
choose an open neighborhood $V$ of $\nu$ satisfying $\mid\Phi(\gamma)-\Phi
(\nu)\mid\leq\varepsilon$ for $\gamma\in V$. Applying Lemma \ref{abstract_ldp}
we find a neighborhood $U$ of $\nu$ in $V$ satisfying \eqref {ldp_upper}. As
\[
Z_{N}\geq2^{-N}\operatorname{exp}\left[  N\operatorname{inf}_{\gamma\in U}%
\Phi(\gamma)\right]  M_{N}(U)\geq2^{-N}\mathrm{e}^{N\left(  \Phi
(\gamma)-\varepsilon\right)  }M_{N}(U),
\]
we get from Lemma \ref{abstract_ldp} that ${\mathbb{P}}$-a.s. one has
eventually
\begin{align*}
Z_{N}  & \geq2^{-N}\operatorname{exp}\left[  N\left(  \Phi(\gamma
)-\varepsilon\right)  \right]  \operatorname{exp}\left[  N\left(  \log
2-H(\nu\mid\mu)-\varepsilon\right)  \right] \\
& \geq\operatorname{exp}\left[  N\left\{  \Phi(\nu)-H(\nu\mid\mu
)-2\varepsilon\right\}  \right]  ,
\end{align*}
and therefore
\[
\liminf_{N\rightarrow\infty}{\frac{1}{N}}\log Z_{N}\geq\Phi(\nu)-H(\nu\mid\mu)
\]
almost surely, for all $\nu$ with $H(\nu\mid\mu)\leq\log2$. This proves the
lower bound.

We now prove the upper bound. We use the well-known fact that there exists a
compact set $K\subset{{\mathcal{M}}}_{1}^{+}(S)$ such that ${\mathbb{P}%
}\left(  L_{N}\notin K\right)  \leq\operatorname{exp}\left[  -N(\log
2+1)\right]  $. Let $D_{N}$ be the event
\[
D_{N}\overset{{\text{def}}}{=}\bigcap_{\alpha=1}^{2^{N}}\left\{  L_{N,\alpha
}\in K\right\}  .
\]
Then
\[
{\mathbb{P}}\left[  D_{N}^{c}\right]  \leq2^{N}2\mathrm{e}^{-N\left(
\log2+1\right)  },
\]
and therefore
\[
{\mathbb{P}}\left[  \liminf_{N\rightarrow\infty}D_{N}\right]  =1.
\]
Fix $\varepsilon>0$. For any $\nu\in K$, we choose $V_{\nu}$ such that
$\Phi(\gamma)-\Phi(\nu)\leq\varepsilon$ for $\gamma\in V_{\nu}$, and then
$U_{\nu}\subset V_{\nu}$ according to Lemma \ref{abstract_ldp}. The $U_{\nu}$
cover $K$, and we choose a finite subcover, call it $U_{\nu_{1}},\dots
,U_{\nu_{m}}$. Then, on $D\overset{{\text{def}}}{=}\liminf_{N}D_{N}$ we have,
writing $U_{k}$ instead of $U_{\nu_{k}}$,
\begin{align*}
Z_{N}  & =2^{-N}\sum_{\alpha}\operatorname{exp}\left[  N\Phi(L_{N,\alpha
})\right] \\
& =2^{-N}\sum_{k=1}^{m}\sum_{\alpha:L_{N,\alpha}\in U_{k}}\operatorname{exp}%
\left[  N\Phi(L_{N,\alpha})\right] \\
& \leq2^{-N}\sum_{k=1}^{m}\operatorname{exp}\left[  N\left\{  \Phi(\nu
_{k})+\varepsilon\right\}  \right]  M_{N}(U)\\
& \leq\sum_{k:H(\nu_{k}\mid\mu)\leq\log2}\operatorname{exp}\left[  N\left\{
\Phi(\nu_{k})-H(\nu_{k}\mid\mu)+2\varepsilon\right\}  \right]
\end{align*}
outside an event which has probability at most $m\operatorname{exp}\left[
-N\min_{j\leq m}\delta_{j}\right]  $, where the $\delta_{j}$ corresponds to
the $U_{j}$. From this estimate on gets that ${\mathbb{P}}$-a.s. one has
\[
\limsup_{N\rightarrow\infty}{\frac{1}{N}}\log Z_{N}\leq\sup_{\nu:H(\nu\mid
\mu)\leq\log2}\left[  \Phi(\nu)-H(\nu\mid\mu)\right]  ,
\]
which together with the lower bound settles the proof of Theorem
\ref{main_theorem_abstract}.
\end{proof}

\begin{proof}
[Proof of Theorem \ref{solving_gvp}]To prove (\ref{VarFormula_LinearUnbounded}%
), we cannot directly apply Theorem \ref{main_theorem_abstract} unless $\phi$
is bounded. We therefore truncate $\phi$ by defining $\phi_{M}\left(
x\right)  \overset{\mathrm{def}}{=}\min\left(  M,\max\left(  \phi\left(
x\right)  ,-M\right)  \right)  ,$ $M>0, $ which is bounded and continuous. As
a consequence of our assumption (\ref{Cond_ExponentialMoment}), we have that
for any $\varepsilon>0,$ and $K>0,$ we can find $M$ such that%
\[
\mathbb{P}\left(  \left\vert \sum\nolimits_{i}\left(  \phi\left(  X_{\alpha
,i}\right)  -\phi_{M}\left(  X_{\alpha,i}\right)  \right)  \right\vert
\geq\varepsilon N\right)  \leq\exp\left[  -KN\right]  .
\]
If we choose $K>\log2,$ then with probability going to $0$ exponentially fast
in $N,$ there is no $\alpha$ such that $\left\vert \sum\nolimits_{i}\left(
\phi\left(  X_{\alpha,i}\right)  -\phi_{M}\left(  X_{\alpha,i}\right)
\right)  \right\vert \geq\varepsilon N.$ In particular%
\begin{align*}
\sum_{\alpha}\exp\left[  -\sum\nolimits_{i}\phi_{M}\left(  X_{\alpha
,i}\right)  -\varepsilon N\right]   & \leq\sum_{\alpha}\exp\left[
-\sum\nolimits_{i}\phi\left(  X_{\alpha,i}\right)  \right] \\
& \leq\sum_{\alpha}\exp\left[  -\sum\nolimits_{i}\phi_{M}\left(  X_{\alpha
,i}\right)  +\varepsilon N\right]  .
\end{align*}
Applying Theorem \ref{main_theorem_abstract} to $\phi_{M},$ and passing to the
$M\rightarrow\infty$ limit in the end, proves
(\ref{VarFormula_LinearUnbounded}).

The more complicated claim is the one on the characterization of the maximizer.

We first restrict the analysis of the GVP (\ref{gvp}) to measures $G_{m}$ of
the form (\ref{Def_Gm}). Restricting to the variational formula to these
measures evidently yields a lower bound to the GVP, which actually reads
\begin{equation}
\sup_{m\in{\mathbb{R}}}\Big\{(1-m)\Gamma_{\phi}^{\prime}(m)+\Gamma_{\phi
}(m):m\Gamma_{\phi}^{\prime}(m)-\Gamma_{\phi}(m)\leq\log2\Big\}.\label{gibbs}%
\end{equation}

We now claim that the target function $(1-m)\Gamma_{\phi}^{\prime}%
(m)+\Gamma_{\phi}(m)$ is increasing on $m\in(-\infty,1]$ and decreasing
otherwise; in fact
\[
{\frac{d}{dm}}\left[  (1-m)\Gamma_{\phi}^{\prime}(m)+\Gamma_{\phi}(m)\right]
=(1-m)\Gamma_{\phi}^{\prime\prime}(m)
\]
and $\Gamma_{\phi}^{\prime\prime}(m)>0$, $\forall m\in{\mathbb{R}}$. Thus, we
can restrict the search for a maximizing $m\in{\mathbb{R}}$ in (\ref{gibbs})
to:
\begin{equation}
\sup_{m\in(-\infty,1]}\Big\{(1-m)\Gamma_{\phi}^{\prime}(m)+\Gamma_{\phi
}(m):m\Gamma_{\phi}^{\prime}(m)-\Gamma_{\phi}(m)\leq\log
2\Big\}\label{gibbstre}%
\end{equation}
But monotonicity also implies that the (global) maximum is attained in $m=1$,
i.e. equals $\Gamma_{\phi}\left(  1\right)  $ as long as the side condition
$\Gamma_{\phi}^{\prime}(1)-\Gamma_{\phi}(1)\leq\log2$, i.e. we are in the high
temperature region.

In the low temperature region, i.e. if $\Gamma_{\phi}^{\prime}(1)-\Gamma
_{\phi}(1)>\log2$, we first observe that the function $m\rightarrow
m\Gamma_{\phi}^{\prime}(m)-\Gamma_{\phi}(m)$ is also increasing, this time for
any value of $m\geq0$:
\[
{\frac{d}{dm}}\Big (m\Gamma_{\phi}^{\prime}(m)-\Gamma_{\phi}(m)\Big )=m\Gamma
_{\phi}^{\prime\prime}(m).
\]
Hence, monotonicity of both target and constraint function yields that the
maximum is achieved at the largest possible value, which is the one
satisfying:
\begin{equation}
m\Gamma_{\phi}^{\prime}(m)-\Gamma_{\phi}(m)=\log2.\label{equation_m}%
\end{equation}
As the left-hand side is $0<\log2$ at $m=0,$ and $>\log2$ at $m=1,$ continuity
and strict monotonicity implies that there is a unique $m_{\star}\in(0,1)$
satisfying this equation.

It remains to prove that any maximizer of (\ref{gvp}) is $G_{m_{\ast}}.$ For
an arbitrary probability measure $\nu$ on $S,$ we set%
\[
\psi\left(  \nu\right)  \overset{\mathrm{def}}{=}E_{\nu}\left(  \phi\right)
-H(\nu\mid\mu).
\]
We compute the entropy of $\nu$ relative to $G_{m}$:
\begin{align*}
H\left(  \nu\mid G_{m}\right)   & =E_{\nu}\log{\frac{d\nu}{dG_{m}}}={E}_{\nu
}\log\left(  {\frac{d\nu}{d\mu}}\cdot{\frac{d\mu}{dG_{m}}}\right) \\
& =H\left(  \nu\mid\mu\right)  -E_{\nu}\log{\frac{dG_{m}}{d\mu}}\\
& =H(\nu\mid\mu)-mE_{\nu}\phi+\log Z(m)\\
& =H(\nu\mid\mu)-mE_{\nu}\phi+mE_{G_{m}}\phi-H(G_{m}\mid\mu)
\end{align*}
the last equality stemming from the definition of the $G_{m}$, according to
which
\[
H(G_{m}\mid\mu)=mE_{G_{m}}[\phi]-\log Z(m)
\]
An elementary computation yields%
\[
m\left(  \psi\left(  G_{m}\right)  -\psi\left(  \nu\right)  \right)  =H\left(
\nu\mid G_{m}\right)  +\left(  1-m\right)  \left[  H(G_{m}\mid\mu)-H(\nu
\mid\mu)\right]  .
\]
This is true for any $m,$ but we apply in now to $m=m_{\ast}.$ Then either
$1-m_{\ast}=0,$ or $H(G_{m_{\ast}}\mid\mu)=\log2$. Therefore, if $H(\nu\mid
\mu)\leq\log2,$ then the right hand side above is non-negative, implying that
$\psi\left(  \nu\right)  \leq\psi\left(  G_{m_{\ast}}\right)  ,$ and equality
only if $H\left(  \nu\mid G_{m_{\ast}}\right)  =0,$ i.e. $\nu=G_{m_{\ast}}.$
This proves the claim.
\end{proof}

\begin{proof}
[Proof of Theorem \ref{equiv_PVP_GVP}]To prove the equivalence of Parisi and
Gibbs Variational Principles, we consider the function $\Psi(m)\overset
{\mathrm{def}}{=}{\frac{\log2}{m}}+{\frac{1}{m}}\log E_{\mu}\left[
\mathrm{e}^{m\phi}\right]  -\log2$. Recall that Parisi variational principle
amounts to minimize $\Psi(m)$ over $[0,1]$. Clearly, either is the minimum
attained on the boundary value $m=1$, or in $m$ solution of $\Psi^{\prime
}(m)=0$. In case the optimal value is attained in $m=1$, one immediately sees
that $\operatorname{inf}_{m\in\lbrack0,1]}\Psi(m)=\Psi(1)=\Gamma_{\phi}(1)$,
thus exactly as in the high-temperature case of Theorem \ref{solving_gvp}.
Otherwise, it is crucial to remark that
\[
\Psi^{\prime}(m)={\frac{1}{m^{2}}}\left\{  H(G_{m}\mid\mu)-\log2\right\}  .
\]
Therefore, $\Psi^{\prime}(m)=0$ iff $H(G_{m}\mid\mu)=\log2$. By Theorem
\ref{solving_gvp}, we already know that the generalized Gibbs measure
associated to the solution of the latter equation is optimal. It is also a
simple algebraic exercise to check that with $m_{\star}$ such that
$\Psi^{\prime}(m_{\star})=0$ one also has $\Psi(m_{\star})=\Gamma_{\phi
}^{\prime}(m_{\star})-\log2$, showing the equivalence of Parisi and Gibbs
variational principle in the low temperature case as well.
\end{proof}

\begin{remark}
\label{interpretation} The above considerations also show that the order
parameter of the Parisi formulation, the optimal $m_{\star}$ (with either
$m_{\star}=1$ or such that $\Psi(m_{\star})=0$), is in fact the inverse of
temperature of the generalized Gibbs measure solving the Gibbs variational
principle. Moreover, derivatives of the target function $\Psi$ in the Parisi
formulation naturally appear in terms of entropies of the generalized Gibbs
measures relative to the underlying random media.
\end{remark}

\subsection{The Gibbs measure of spin glasses of
REM-type\label{SubSect_ProofGibbs}}

Let $m_{\star}$ be as defined in Theorem \ref{solving_gvp}, and $G=G_{m_{\star
}}$ the associated measure. We write $v^{2}:=\operatorname*{var}_{G}\left(
\phi\right)  $ for the variance of $\phi$ under $G.$ Then define%
\begin{equation}
a_{N}\overset{\mathrm{def}}{=}E_{G}(\phi)N+\omega(N),\label{Def_a_N}%
\end{equation}
where $\omega(N)\overset{\mathrm{def}}{=}-{\frac{1}{m_{\star}}}\log\sqrt{2\pi
v^{2}N}.$ For $\alpha\in\{1,\dots,2^{N}\}$ let us also abbreviate
$H_{N}(\alpha)\overset{{\text{def}}}{=}\sum_{i=1}^{N}\phi(X_{\alpha,i})$.
\newline

We begin with a technical result:

\begin{lemma}
\label{clt_zero} Assume that the measure $\mu\phi^{-1}$ on $\mathbb{R}$ is
non-lattice. Then, with the above notations, for any $t\in\mathbb{R}$%
\begin{equation}
\lim_{N\rightarrow\infty}2^{N}{\mathbb{P}}\left[  \sum_{i=1}^{N}\phi
(X_{1,i})-a_{N}\geq t\right]  ={\frac{1}{m_{\star}}}\mathrm{e}^{-m_{\star}%
t}.\label{Exponential}%
\end{equation}

\end{lemma}

\begin{proof}
We use the usual transformation of measure argument, writing the probability
on the left-hand side of (\ref{Exponential}) in terms of a new sequence
$\left\{  \widetilde{X}_{i}\right\}  $ of independent random variables with
distribution function $G.$ As $G$ is equivalent to $\mu,$ also $\phi\left(
\widetilde{X}_{i}\right)  $ is non-lattice. We write $G_{N}$ for the
distribution of%
\[
\sum_{i=1}^{N}\left(  \phi(\widetilde{X}_{i})-E_{G}(\phi)\right)  ,
\]
and $\hat{G}_{N}$ for the standardized one: $\hat{G}_{N}\left(  \cdot\right)
=G_{N}\left(  v\sqrt{N}\cdot\right)  .$

By change of measure and integration by parts, it holds:
\[
{\mathbb{P}}\left[  \sum\nolimits_{i=1}^{N}\phi(X_{1,i})-a_{N}\geq t\right]
=\operatorname{exp}N\left[  \log E[\mathrm{e}^{m_{\star}\phi}]-m_{\star}%
E_{G}(\phi)\right]  \int_{t+\omega(N)}\mathrm{e}^{-m_{\star}y}G_{N}(dy)
\]
Recall that in low temperature,%
\[
\log E\left(  \mathrm{e}^{m_{\ast}\phi}\right)  -m_{\ast}E_{G}\left(
\phi\right)  =-\log2,
\]
and therefore%
\begin{align*}
2^{N}{\mathbb{P}}\left[  \sum\nolimits_{i=1}^{N}\phi(X_{1,i})-a_{N}\geq
t\right]   & =\int_{t+\omega(N)}\mathrm{e}^{-m_{\star}y}G_{N}(dy)\\
& =\int_{t+\omega(N)}\int_{m_{\star}y}^{\infty}\mathrm{e}^{-z}dzG_{N}(dy)\\
& =\int_{m_{\ast}\left(  t+\omega\left(  N\right)  \right)  }^{\infty
}dz\mathrm{e}^{-z}\hat{G}_{N}\left(  \left(  \frac{t+\omega(N)}{\sqrt{N}%
v},\frac{z}{\sqrt{N}vm_{\ast}}\right]  \right)  .
\end{align*}

We use now the assumption that $\phi(\widetilde{X}_{i})$ has a non-lattice
distribution. This implies that we can approximate the distribution function
of $\hat{G}_{N}$ by a smooth distribution function, up to an error of order
$o\left(  N^{-1/2}\right)  .$ More precisely, if we define
$\operatorname{GaussDF}$ to be the distribution function of the standard
normal distribution, and $\varphi$ the density, then, uniformly in $x,$%
\begin{equation}
\hat{G}_{N}\left(  \left(  -\infty,x\right]  \right)  =\operatorname{GaussDF}%
\left(  x\right)  -\frac{\mu_{3}}{6v^{3}\sqrt{N}}\left(  x^{2}-1\right)
\varphi\left(  x\right)  +o\left(  N^{-1/2}\right)  ,\label{Edgeworth}%
\end{equation}
where $\mu_{3}$ is the third moment of $\phi(\widetilde{X}_{i})-E_{G}(\phi).$
(see e.g. \cite{fellerII}, Theorem XVI.4.1). If we replace $\hat{G}_{N}$ above
by the standard normal distribution, and transform back, we get an expression%
\[
\operatorname{exp}\left(  {\frac{m_{\star}^{2}}{2}}Nv\right)  \int_{t}%
^{\infty}\operatorname{exp}\left[  -{\frac{1}{2Nv}}\left(  z-m_{\star
}Nv+\omega(N)\right)  ^{2}\right]  {\frac{dz}{\sqrt{2\pi Nv}}}=\int
_{t}^{\infty}\mathrm{e}^{-m_{\star}y+o(1)}dy.
\]
In order to prove the lemma, it therefore suffices to prove that the two error
summands in (\ref{Edgeworth}) contribute nothing in the $N\rightarrow\infty$
limit. This is evident for $o\left(  N^{-1/2}\right)  $ as%
\[
\int_{m_{\ast}\left(  t+\omega\left(  N\right)  \right)  }^{\infty
}dz\mathrm{e}^{-z}=O\left(  \sqrt{N}\right)  .
\]
For the middle Edgeworth term in (\ref{Edgeworth}), the special form is of no
importance, and we only use that it is of the form $h\left(  x\right)
/\sqrt{N}$ with a bounded smooth function $h:$%
\[
\int_{m_{\ast}\left(  t+\omega\left(  N\right)  \right)  }^{\infty
}dz\mathrm{e}^{-z}\left[  h\left(  \frac{z}{\sqrt{N}vm_{\ast}}\right)
-h\left(  \frac{t+\omega(N)}{\sqrt{N}v}\right)  \right]  =O\left(  1\right)  .
\]

\end{proof}

\begin{proposition}
\label{pp_abstract} Within the above setting:

\begin{itemize}
\item[a)] The point process $\sum_{\alpha}\delta_{H_{N}(\alpha)-a_{N}}$
converges weakly to a Poisson point process with intensity measure
$\mathrm{e}^{-m_{\ast}t}dt.$

\item[b)] The point process $\sum_{\alpha}\delta_{\exp\left[  H_{N}%
(\alpha)-a_{N}\right]  }$ converges weakly to a $\operatorname{PP}\left(
m_{\ast}\right)  .$
\end{itemize}
\end{proposition}

\begin{proof}
b) is evident from a). For the first claim, we shall exploit the equivalence
of weak convergence and convergence of Laplace functionals. For a continuous
function with compact support $F\in C_{o}({\mathbb{R}})$ we have:
\begin{align*}
{\mathbb{E}}\mathrm{e}^{-\sum\limits_{\alpha}F(H_{N}(\alpha)-a_{N})}  &
=\left\{  {\mathbb{E}}\mathrm{e}^{-F\left(  H_{N}(1)-a_{N}\right)  }\right\}
^{2^{N}}\\
& =\left\{  1-{\mathbb{E}}\left[  1-\mathrm{e}^{-F\left(  H_{N}(1)-a_{N}%
\right)  }\right]  \right\}  ^{2^{N}}.
\end{align*}
By Lemma \ref{clt_zero}, this converges to $\operatorname{exp}\left[
-\int\left(  1-\mathrm{e}^{-F(t)}\right)  \mathrm{e}^{-m_{\star}t}dt\right]
$, settling a).
\end{proof}

Since
\[
{{\mathcal{G}}}_{N,\phi}(\alpha)=\frac{\exp\left[  H_{N}(\alpha)-a_{N}\right]
}{\sum_{\alpha^{\prime}}\operatorname{exp}\left[  H_{N}(\alpha^{\prime}%
)-a_{N}\right]  },
\]
to prove Theorem \ref{limiting_gibbs_measure} it suffices to prove that in low
temperature the normalization commutes with taking the $N\rightarrow\infty$
limit. For this, we have the following:

\begin{lemma}
\label{gan} Suppose $\phi$ is such that the system is in low temperature, and
let $\varepsilon>0$. There exists $C>0$ such that
\[
{\mathbb{P}}\left[  \sum\nolimits_{\alpha}\operatorname{exp}\left[
H_{N}(\alpha)-a_{N}\right]  1_{H_{N}(\alpha)-a_{N}\notin\lbrack-C,C]}%
\geq\varepsilon\right]  \leq\varepsilon,
\]
for large enough $N$.
\end{lemma}

\begin{proof}
First, by Lemma \ref{clt_zero} we clearly have that
\begin{align*}
{\mathbb{P}}\left[  \sum\nolimits_{\alpha}\mathrm{e}^{H_{N}(\alpha)-a_{N}%
}1_{H_{N}(\alpha)-a_{N}\geq C}\geq\varepsilon\right]   & \leq{\mathbb{P}%
}\left[  \exists\alpha\in\Sigma_{N}:H_{N}(\alpha)-a_{N}\geq C\right] \\
& \leq2^{N}{\mathbb{P}}\left[  H_{N}(1)-a_{N}\geq C\right]  \leq
\operatorname{const}\times\mathrm{e}^{-m_{\star}C}%
\end{align*}
for large enough $N,$ which can be made arbitrarily small by choosing $C$
large enough. So, it remains to prove that we can find $C$ such that
\[
{\mathbb{P}}\left[  \sum\nolimits_{\alpha}\mathrm{e}^{H_{N}(\alpha)-a_{N}%
}1_{H_{N}(\alpha)-a_{N}\leq-C}\geq\varepsilon\right]  \leq{\frac{\varepsilon
}{2}}.
\]
To see the last inequality, remark that the left hand side is bounded by%
\[
\frac{2^{N}}{\varepsilon}\mathbb{E}\left[  \mathrm{e}^{H_{N}(1)-a_{N}}%
1_{H_{N}(\alpha)-a_{N}\leq-C}\right]  .
\]
We proceed along the lines of Lemma \ref{clt_zero}:%
\begin{align*}
& 2^{N}{\mathbb{E}}\left(  \mathrm{e}^{H_{N}(1)-a_{N}}1_{H_{N}(1)-a_{N}\leq
-C}\right) \\
& =2^{N}\mathrm{e}^{N\log E[\operatorname{exp}m\phi]-NE_{G}(\phi)}%
\mathrm{e}^{-\omega(N)}\int\limits_{-\infty}^{-C+\omega(N)}\mathrm{e}%
^{(1-m_{\star})y}G_{N}(dy)\\
& =\mathrm{e}^{-\omega(N)}\int\limits_{-\infty}^{-C+\omega(N)}\mathrm{e}%
^{(1-m_{\star})y}G_{N}(dy).
\end{align*}
We again use (\ref{Edgeworth}). Important is now only that $G_{N},$ up to a
signed measure $R_{N}$ of total variation $o\left(  N^{-1/2}\right)  ,$ is
given by a (signed) measure $h_{N}\left(  y/\sqrt{N}\right)  dy/\sqrt{N},$
where $h_{N}$ bounded, uniformly in $N$. Using that $m_{\star}<1,$ we get for
the smooth part%
\begin{align*}
\mathrm{e}^{-\omega(N)}\int\limits_{-\infty}^{-C+\omega(N)}\mathrm{e}%
^{(1-m_{\star})y}\frac{h_{N}\left(  y/\sqrt{N}\right)  }{\sqrt{N}}dy  &
\leq\operatorname{const}\times\frac{\mathrm{e}^{-\omega(N)}}{\sqrt{N}}%
\int\limits_{-\infty}^{-C+\omega(N)}\mathrm{e}^{(1-m_{\star})y}dy\\
& =\operatorname{const}\times\mathrm{e}^{-C(1-m_{\star})}\frac{1}{(1-m_{\star
})\sqrt{N}}\mathrm{e}^{-m_{\star}\omega(N)}\\
& =\operatorname{const}\times\mathrm{e}^{-C(1-m_{\star})}\frac{\sqrt{2\pi
v^{2}}}{1-m_{\star}}.
\end{align*}
Taking $C$ large, we can make that arbitrarily small. For the $R_{N}$-part, we
have by partial integration and Fubini:%
\begin{align*}
\mathrm{e}^{-\omega(N)}\int\limits_{-\infty}^{-C+\omega(N)}\mathrm{e}%
^{(1-m_{\star})y}R_{N}(dy)  & =\int\limits_{-\infty}^{-C+\omega(N)}%
R_{N}(dy)\int\limits_{-\infty}^{(1-m_{\star})y}dz\mathrm{e}^{z}\\
& =\mathrm{e}^{-\omega(N)}\int\limits_{-\infty}^{(1-m_{\star})\left(
-C+\omega\left(  N\right)  \right)  }dz\mathrm{e}^{z}\\
& \times R_{N}\left(  \left(  \frac{z}{1-m_{\star}},-C+\omega(N)\right]
\right) \\
& =o\left(  N^{-1/2}\right)  \mathrm{e}^{-\omega(N)}\exp\left[  (1-m_{\star
})\left(  -C+\omega\left(  N\right)  \right)  \right] \\
& =o\left(  1\right)  .
\end{align*}

\end{proof}

\begin{proof}
[Proof of Theorem \ref{limiting_gibbs_measure}]We denote by $M$ the space of
Radon measures on $(0,\infty)$ endowed with the vague topology. By
${{\mathcal{H}}}_{N}$ we denote the point process associated to the collection
of points $(\operatorname{exp}[H_{N}(\alpha)-a_{N}],\alpha\in\Sigma_{N})$ and
${\mathcal{H}}$ be its weak limit. We choose a continuous function
$h:{\mathbb{R}}_{+}\rightarrow{\mathbb{R}}_{+}$ with $h(x)=x$ for $x\in
\lbrack1/C,C]$, $h(x)\leq x\;\forall x$ and $h(x)=0$ for $x\notin
\lbrack1/2C,2C]$. Then $\int h\,d{{\mathcal{H}}}_{N}$ converges weakly to
$\int h\,d{\mathcal{H}}$ by continuity of the mapping $M\ni\Xi\rightarrow
\int\,hd\Xi$. By Lemma \ref{gan}, to $\varepsilon>0$ we can find $C>0$ large
enough such that
\[
{\mathbb{P}}\left[  \int_{0}^{1/C}xd\,{{\mathcal{H}}}_{N}+\int_{C}^{\infty
}xd\,{{\mathcal{H}}}_{N}\geq\varepsilon\right]  \leq\varepsilon,
\]
uniformly in $N$, from which we see by approximation that $\int_{0}^{\infty
}xd\,{\mathcal{H}}_{N}$ converges weakly to $\int_{0}^{\infty}xd\,{\mathcal{H}%
}$. This also implies that $({\mathcal{H}}_{N},\int_{0}^{\infty}%
xd\,{\mathcal{H}}_{N})$ converges weakly towards $({\mathcal{H}},\int
_{0}^{\infty}x\,d{\mathcal{H}})$. Theorem \ref{limiting_gibbs_measure} then
clearly follows from the continuity of the mapping $M\times(0,\infty
)\rightarrow M$ defined through $(\Xi,a)\mapsto\Xi\theta_{a}^{-1}$ with
$\theta_{a}:(0,\infty)\rightarrow(0,\infty)$ and $\theta_{a}(x)\overset
{{\text{def}}}{=}x/a$.
\end{proof}

\subsection{The free energy of the REM+Cavity}

\begin{proof}
[Proof of Proposition \ref{propvariational}]Performing the trace over the
Ising-spins we have:
\begin{equation}
f_{N}(\beta)={\frac{1}{N}}\log2^{-N}\sum_{\alpha}\operatorname{exp}\left[
{\beta X_{\alpha}+\sum_{i}\log\cosh(\beta g_{\alpha,i})}\right]
,\label{sumout}%
\end{equation}
and we write the $X_{\alpha}$'s as the sum of $N$ independent standard
Gaussians $X_{\alpha,i}$, i.e. $X_{\alpha}=\sum_{i=1}^{N}X_{\alpha,i}$. By
Theorems \ref{solving_gvp} and \ref{equiv_PVP_GVP} applies to the function
$\phi(x_{1},x_{2})\overset{{\text{def}}}{=}\beta x_{1}+\log\cosh(\beta x_{2})$
and $\mu(x_{1},x_{2})$ a standard, bivariate Gaussian.

The high temperature region $\Gamma_{\phi}^{\prime}(1)-\Gamma_{\phi}%
(1)\leq\log2$ is equivalent to%
\[
E\left[  \cosh(\beta g)\log\cosh(\beta g)\right]  \leq\mathrm{e}^{\beta^{2}%
/2}\log2,
\]
$g$ a standard normal, and $E$ here the expectation with respect to $g.$ So we
have to identify this region. We prove that there is a unique $\beta
_{\mathrm{cr}}>0$ such that this inequality holds if and only if $\beta
\leq\beta_{\mathrm{cr}}.$ To prove that, consider%
\[
H(\beta)\overset{\mathrm{def}}{=}\mathrm{e}^{-\beta^{2}/2}E[\log\cosh(\beta
g)\cosh(\beta g)].
\]
Then%
\begin{align*}
{\frac{dH}{d\beta}}  & =-\beta\mathrm{e}^{-\beta^{2}/2}E[\cosh(\beta
g)\log\cosh(\beta g)]+\mathrm{e}^{-\beta^{2}/2}E[g\sinh(\beta g)]\\
& +\mathrm{e}^{-\beta^{2}/2}E[g\log\cosh(\beta g)\sinh(\beta g)]\\
& =\beta\mathrm{e}^{-\beta^{2}/2}\left\{  E[\cosh(\beta g)]+E[\sinh^{2}(\beta
g)\cosh(\beta g)^{-1}]\right\}  ,
\end{align*}
by Gaussian partial integration. So the derivative of $H$ is positive. It is
easy to see that $\lim_{\beta\rightarrow\infty}H\left(  \beta\right)
=\infty.$ So it follows that there is a unique $\beta_{\mathrm{cr}}>0$ such
that $H\left(  \beta_{\mathrm{cr}}\right)  =\log2,$ and $H\left(
\beta\right)  \leq\log2$ if and only if $\beta\leq\beta_{\mathrm{cr}}.$

If $\beta\leq\beta_{\mathrm{cr}},$ then $f\left(  \beta\right)  =\log E_{\mu
}\mathrm{e}^{m\phi}=\beta^{2}/2.$ If $\beta>\beta_{\mathrm{cr}},$ then we have
to determine $m_{\ast}$ according to (\ref{entropy_condition}) which gives
(\ref{entropy_condition_remcavity}), and plug it into (\ref{FreeEnergy_Parisi}%
) which gives the expression for the free energy (\ref{FreeEnergy_REMC}).
\end{proof}

\subsection{The Gibbs measure of the REM+Cavity\label{SubSect_GibbsREMC}}

\begin{proof}
[Proof of Proposition \ref{point_process}]Performing the trace over the Ising
spins, the Gibbs weight of the pure state $\alpha\in\{1,\dots,2^{N}\}$ in the
REM+Cavity reads
\[
{{\mathcal{G}}}_{\beta,N}^{\left(  1\right)  }(\alpha)=\operatorname{exp}%
\left[  \beta X_{\alpha}+\sum_{i=1}^{N}\log\cosh(\beta g_{\alpha,i})\right]
\Bigg /\sum_{\alpha^{\prime}}\operatorname{exp}\left[  \beta X_{\alpha
^{\prime}}+\sum_{i=1}^{N}\log\cosh(\beta g_{\alpha^{\prime},i}).\right]
\]
As in the proof of the Proposition \ref{propvariational}, we may then replace
the $X_{\alpha}$ with $\sum_{i=1}^{N}X_{\alpha,i}$ for a double sequence of
independent standard Gaussians $X_{\alpha,i}$, and then we apply Theorem
\ref{limiting_gibbs_measure}.
\end{proof}

For the proof of the other results, we need some remarkable properties of the
point processes $\operatorname{PP}\left(  m\right)  $.

\begin{lemma}
\label{formulas_rem} Assume that $\left\{  v_{\alpha}\right\}  _{\alpha
\in\mathbb{N}}$ are the points of a $\operatorname{PP}\left(  m\right)  $.
Consider also, independently of this point process, a sequence $\left\{
(U_{\alpha},V_{\alpha})\right\}  _{\alpha\in\mathbb{N}}$ of i.i.d. two
dimensional square integrable random vectors satisfying $V_{\alpha}\geq1.$
Then the following formulas hold:%
\begin{equation}
E\left[  {\frac{\sum_{\alpha}v_{\alpha}U_{\alpha}}{\sum_{\alpha}v_{\alpha
}V_{\alpha}}}\right]  ={\frac{E\left[  UV^{m_{\star}-1}\right]  }{E\left[
V^{m_{\star}}\right]  },}\label{Tala1}%
\end{equation}%
\begin{equation}
E\left[  {\frac{\sum_{\alpha\neq\beta}v_{\alpha}v_{\beta}U_{\alpha}U_{\beta}%
}{\left(  \sum_{\alpha}v_{\alpha}V_{\alpha}\right)  ^{2}}}\right]  =m_{\star
}\left(  {\frac{E\left[  UV^{m_{\star}-1}\right]  }{E\left[  V^{m_{\star}%
}\right]  }}\right)  ^{2},\label{Tala2}%
\end{equation}%
\begin{equation}
E\left[  {\frac{\sum_{\alpha}v_{\alpha}^{2}U_{\alpha}^{2}}{\left(
\sum_{\alpha}v_{\alpha}V_{\alpha}\right)  ^{2}}}\right]  =(1-m_{\star}%
){\frac{E\left[  U^{2}V^{m_{\star}-2}\right]  }{E\left[  V^{m_{\star}}\right]
}.}\label{Tala3}%
\end{equation}

\end{lemma}

For a proof, see \cite[Theorem 6.4.5]{talagrand}.

\begin{proof}
[Proof of Proposition \ref{ultrametricity}]We fix some notation. Let
\[
w_{\alpha}^{(1,2)}=\operatorname{exp}\left(  \beta X_{\alpha}+\sum_{i=3}%
^{N}\log\cosh(\beta g_{\alpha,i})-a_{N}\right)  ,
\]
stand for the (not normalized) Boltzmann weight of the pure state $\alpha$
with a cavity in the sites $i=1,2$, and the \textquotedblleft centering
constant\textquotedblright\ being given by (\ref{Def_a_N}) specialized to the
present setting. Remark that these weights somewhat differ from the original
ones without the cavity, but we clearly still have, in total analogy with
Proposition \ref{pp_abstract}, weak convergence of $(w_{\alpha}^{(1,2)})$
towards a collection $(v_{\alpha})$ distributed according to a
$\operatorname{PP}\left(  m_{\ast}\right)  .$

For the proof of claim (\ref{overlap1}), expanding the quadratic terms, by
symmetry and obvious bounds we have:%
\begin{align*}
{\mathbb{E}}\left[  \left\langle \delta_{\alpha=\alpha^{\prime}}\left(
q(\sigma,\sigma^{\prime})-q_{\star}\right)  ^{2}\right\rangle _{\beta
,N}^{\otimes2}\right]   & \sim{\mathbb{E}}\left[  \left\langle \delta
_{\alpha=\alpha^{\prime}}\sigma_{1}\sigma_{2}\sigma_{1}^{\prime}\sigma
_{2}^{\prime}\right\rangle _{\beta,N}^{\otimes2}\right] \\
& -2q_{\star}{\mathbb{E}}\left[  \left\langle \delta_{\alpha=\alpha^{\prime}%
}\sigma_{1}\sigma_{1}^{\prime}\right\rangle _{\beta,N}^{\otimes2}\right]
+q_{\star}^{2}{\mathbb{E}}\left\langle \delta_{\alpha,\alpha^{\prime}%
}\right\rangle _{\beta,N}^{\otimes}\\
& =A_{1}-A_{2}+A_{3},\ \mathrm{say.}%
\end{align*}
$\sim$ meaning that the quotient converges to $1,$ as $N\rightarrow\infty. $

The point process $({{\mathcal{G}}}_{\beta,N}^{\left(  1\right)  }%
(\alpha))_{\alpha}$ converges to a $\operatorname{PD}\left(  m_{\ast}\right)
$ which implies%
\begin{equation}
A_{3}={\mathbb{E}}\left[  \sum_{\alpha}{{\mathcal{G}}}_{\beta,N}^{\left(
1\right)  }(\alpha)^{2}\right]  \sim1-m_{\star}.\label{one}%
\end{equation}
For $A_{1},$ we get%
\begin{align*}
A_{1}  & ={\mathbb{E}}\left[  {\frac{\sum_{\alpha}\left\{  \sinh(\beta
g_{\alpha,1})\sinh(\beta g_{\alpha,2})\right\}  ^{2}\left(  w_{\alpha}%
^{(1,2)}\right)  ^{2}}{\left(  \sum_{\alpha}\cosh(\beta g_{\alpha,1}%
)\cosh(\beta g_{\alpha,2})w_{\alpha}^{(1,2)}\right)  ^{2}}}\right] \\
& ={\mathbb{E}}\left[  {\frac{\sum_{\alpha}\left\{  \sinh(\beta g_{\alpha
,1})\sinh(\beta g_{\alpha,2})\right\}  ^{2}\left(  \overline{w}_{\alpha
}^{(1,2)}\right)  ^{2}}{\left(  \sum_{\alpha}\cosh(\beta g_{\alpha,1}%
)\cosh(\beta g_{\alpha,2})\overline{w}_{\alpha}^{(1,2)}\right)  ^{2}}}\right]
,
\end{align*}
where $\overline{w}_{\alpha}^{(1,2)}\overset{\mathrm{def}}{=}w_{\alpha
}^{(1,2)}\left/  \sum_{\alpha}w_{\alpha}^{(1,2)}\right.  .$ The corresponding
point process converges weakly to $\operatorname{PD}\left(  m_{\ast}\right)
,$ and taking $U_{\alpha}=\sinh(\beta g_{\alpha,1})\sinh(\beta g_{\alpha
,2}),\;V_{\alpha}=\cosh(\beta g_{\alpha,1})\cosh(\beta g_{\alpha,2}),$ a
simple domination argument shows that one can pass to the $N\rightarrow\infty$
limit, replacing the $\overline{w}_{\alpha}^{(1,2)}$ by the points of this
point process. Applying then (\ref{Tala3}), one gets%
\begin{align*}
A_{1}  & \sim(1-m_{\star}){\frac{E\left[  \tanh^{2}(\beta g_{1})\tanh
^{2}(\beta g_{2})\cosh(\beta g_{1})^{m_{\star}}\cosh(\beta g_{2})^{m_{\star}%
}\right]  }{E\left[  \cosh(\beta g_{1})^{m_{\star}}\cosh(\beta g_{2}%
)^{m_{\star}}\right]  }}\\
& =(1-m_{\star})q_{\star}^{2}.
\end{align*}
In a similar way, one proves%
\[
A_{2}\sim2(1-m_{\star})q_{\star}^{2}.
\]
This proves (\ref{overlap1}). (\ref{overlap2}) follows similarly.
\end{proof}

\begin{proof}
[Proof of Proposition \ref{chaos_temperature}]Consider for the moment the
$\beta$-system only. Denote by $f(\beta)$ the free energy and by $G_{m_{\star
}(\beta)}$ the associated extremal measure solving the GVP. For $\varepsilon
>0$, set $B_{\beta,\varepsilon}\overset{{\text{def}}}{=}B_{\varepsilon
}(G_{m_{\star}(\beta)})\subset{\mathcal{M}}_{1}^{+}({\mathbb{R}}^{2})$ for the
open ball of radius $\varepsilon$ and center $G_{m_{\star}(\beta)}$. For
$\alpha\in\Sigma_{N}$ we denote by $L_{N,\alpha}$ the empirical measures
associated to the free energies of the pure states.

We first claim that given $\varepsilon>0$ there exists $\delta>0$ such that
\begin{equation}
{\mathbb{P}}\left[  {\mathcal{G}}_{\beta,N}\left(  \alpha\in\Sigma
_{N}:L_{N,\alpha}\notin B_{\beta,\varepsilon}\right)  \geq\mathrm{e}^{-\delta
N}\right]  \leq\mathrm{e}^{-\delta N}.\label{empirical_measure_chaos}%
\end{equation}
To see this, first observe that uniqueness of the maximizers solving the GVP
implies that, with $\phi(x_{1},x_{2})=\beta x_{1}+\log\cosh(\beta x_{2})$ and
$\mu$ standard bivariate Gaussian,
\begin{equation}
f(\beta,\varepsilon)\overset{\mathrm{def}}{=}\sup\left\{  E_{\nu}[\phi
]-H(\nu\mid\mu):H(\nu\mid\mu)\leq\log2,\nu\notin B_{\beta,\varepsilon
}\right\}  <f(\beta).
\end{equation}
Using the same argument as in the proof of Theorem \ref{main_theorem_abstract}
we get
\[
\lim_{N\rightarrow\infty}{\frac{1}{N}}\log2^{-N}\sum_{\alpha:L_{N,\alpha
}\notin B_{\beta,\varepsilon}}\mathrm{e}^{\beta X_{\alpha}+\sum_{i=1}^{N}%
\log\cosh(\beta g_{\alpha,i})}\leq f(\beta,\varepsilon),\;{\mathbb{P}}-a.s.
\]
Using the variance estimate from \eqref{variance_estimate} and the Tchebychev
inequality, it is easily seen that
\begin{equation}
{\mathbb{P}}[\Lambda_{N}^{c}(\beta,\varepsilon)]\leq\operatorname{exp}(-\delta
N),\label{very_small}%
\end{equation}
for some $\delta>0$, where%
\begin{align*}
\Lambda_{N}(\beta,\varepsilon)  & =\left\{  2^{-N}\sum_{\alpha:L_{N,\alpha
}\notin B_{\beta,\varepsilon}}\mathrm{e}^{\beta X_{\alpha}+\sum_{i}\log
\cosh(\beta g_{\alpha,i})}\leq f(\beta,\varepsilon)+{\frac{\eta}{3}},\right.
\\
& \left.  2^{-N}\sum_{\alpha}\mathrm{e}^{\beta X_{\alpha}+\sum_{i}\log
\cosh(\beta g_{\alpha,i})}\geq f(\beta,\varepsilon)+{\frac{2}{3}}\eta\right\}
\end{align*}
where $\eta\overset{\mathrm{def}}{=}f(\beta)-f(\beta,\varepsilon).$ This
clearly implies \eqref {empirical_measure_chaos}.

If $\beta\neq\beta^{\prime},$ then we can choose $\varepsilon>0$ such that
$B_{\beta,\varepsilon}\cap B_{\beta^{\prime},\varepsilon}=\emptyset,$ and the
claim a) follows.

As for claim (\ref{chaos2}), we observe that
\begin{equation}
{\mathbb{E}}\left\langle \delta_{\alpha\neq\alpha^{\prime}}q(\sigma
,\sigma')^2\right\rangle _{\beta,\beta^{\prime},N}^{\otimes
2}={\mathbb{E}}\left\langle q(\sigma,\sigma')^{2}\right\rangle
_{\beta,\beta^{\prime},N}^{\otimes2}-{\mathbb{E}}\left\langle \delta
_{\alpha=\alpha^{\prime}}q(\sigma,\sigma')^2\right\rangle _{\beta
,\beta^{\prime},N}^{\otimes2}.\label{chaos_one}%
\end{equation}
Since $q(\sigma,\sigma')^2\leq1$ for all $\sigma,\sigma^{\prime} $, by
claim $a)$ of this Proposition the second term on the right hand side in
\eqref{chaos_one} is in the limit $N\rightarrow\infty$ vanishing. As for the
first term on the right hand side of \eqref {chaos_one}, by symmetry and
obvious bounds we have
\begin{equation}
{\mathbb{E}}\left\langle q(\sigma,\sigma')^2\right\rangle _{\beta
,\beta^{\prime},N}^{\otimes2}=\left\{  1+O(1/N)\right\}  {\mathbb{E}%
}\left\langle \sigma_{1}\sigma_{1}^{\prime}\sigma_{2}\sigma_{2}^{\prime
}\right\rangle _{\beta,\beta^{\prime},N}^{\otimes2}+O(1/N).\label{chaos_two}%
\end{equation}

Let us now set $w_{\alpha}^{(1,2,\beta)}\overset{{\text{def}}}{=}%
\operatorname{exp}\left[  \beta X_{\alpha}+\sum_{i=3}^{N}\log\cosh(\beta
g_{\alpha,i})-a_{N}(\beta)\right]  $ for the Boltzmann weight of the pure
state $\alpha$ with a cavity in the sites $i=1,2$ associated to the $\beta
$-system, and $a_{N}(\beta)$ being the centering constant from (\ref{Def_a_N})
specialized to the setting. Analogously, we write $w_{\alpha}^{(1,2,\beta
^{\prime})}\overset{{\text{def}}}{=}\operatorname{exp}\left[  \beta^{\prime
}X_{\alpha}+\sum_{i=3}^{N}\log\cosh(\beta^{\prime}g_{\alpha,i})-a_{N}%
(\beta^{\prime})\right]  $ in case of the $\beta^{\prime}$-system. With this
notations in mind, we write the expectation on the right hand side of
\eqref {chaos_one} as
\begin{align*}
& {\mathbb{E}}\left\langle \sigma_{1}\sigma_{1}^{\prime}\sigma_{2}\sigma
_{2}^{\prime}\right\rangle _{\beta,\beta^{\prime},N}^{\otimes2}=\\
& ={\mathbb{E}}\left[  {\frac{\sum_{\alpha}\sinh(\beta g_{\alpha,1}%
)\sinh(\beta g_{\alpha,2})w_{\alpha}^{(1,2,\beta)}}{\sum_{\alpha}\cosh(\beta
g_{\alpha,1})\cosh(\beta g_{\alpha,2})w_{\alpha}^{(1,2,\beta)}}}\times
{\frac{\sum_{\alpha^{\prime}}\sinh(\beta^{\prime}g_{\alpha^{\prime},1}%
)\sinh(\beta^{\prime}g_{\alpha^{\prime},2})w_{\alpha^{\prime}}^{(1,2,\beta
^{\prime})}}{\sum_{\alpha^{\prime}}\cosh(\beta^{\prime}g_{\alpha^{\prime}%
,1})\cosh(\beta^{\prime}g_{\alpha^{\prime},2})w_{\alpha^{\prime}}%
^{(1,2,\beta^{\prime})}}}\right]  .
\end{align*}
By Proposition \ref{pp_abstract}.b) the Point Process associated to the
collection of "points" of the $\beta$-system $(w_{\alpha}^{(1,2,\beta)})$
converges weakly to a $\operatorname{PP}\left(  m_{\star}(\beta)\right)  $,
while the Point Process associated to the $\beta^{\prime}$-system converges to
a $\operatorname{PP}\left(  m_{\star}(\beta^{\prime})\right)  $. On the other
hand, using similar arguments as in the proof of claim $a)$ it is not
difficult to see that the limiting point processes are in fact independent.
[Given a compact set $K\subset{\mathbb{R}}_{+}$, the ${\mathbb{P}}%
$-probability to find a configuration $\alpha\in\Sigma_{N}$ such that
$w_{\alpha}^{(1,2,\beta)}\in K$ and simultaneously $w_{\alpha}^{(1,2,\beta
^{\prime})}\in K$ is exponentially small in $N$.] Hence, the right hand side
of \eqref {chaos_two} converges with $N\rightarrow\infty$ to the product
\[
E\left[  {\frac{\sum_{\alpha}\sinh(\beta g_{\alpha,1})\sinh(\beta g_{\alpha
,2})w_{\alpha}}{\sum_{\alpha}\cosh(\beta g_{\alpha,1})\cosh(\beta g_{\alpha
,2})w_{\alpha}}}\right]  \times E\left[  {\frac{\sum_{\alpha}\sinh
(\beta^{\prime}g_{\alpha,1})\sinh(\beta^{\prime}g_{\alpha,2})w_{\alpha
}^{\prime}}{\sum_{\alpha}\cosh(\beta^{\prime}g_{\alpha,1})\cosh(\beta^{\prime
}g_{\alpha,2})w_{\alpha}^{\prime}}}\right]  ,
\]
with $(w_{\alpha})$ a $\operatorname{PP}\left(  m_{\star}(\beta)\right)  $,
and $(w_{\alpha}^{\prime})$ of a $\operatorname{PP}\left(  m_{\star}%
(\beta^{\prime})\right)  $. By (\ref{Tala1}) both expectations are seen to be
equal to zero. This settles claim (\ref{chaos2}) of Proposition
\ref{chaos_temperature}.
\end{proof}

\end{document}